\documentclass[reqno]{amsart}
\usepackage{amssymb,amsmath,amsfonts,amscd,amsthm}
\usepackage{amsbsy,bm}

\newtheorem{theorem}{Theorem}[section]

\newtheorem{main theorem}{Main Theorem}
\newtheorem{lemma}{Lemma}[section]

\newtheorem{corollary}{Corollary}[section]
\theoremstyle{remark}
\newtheorem{remark}{Remark}[section]

\begin{document}
\title[Rigidity of Lagrangian submanifolds in $\mathbb{S}^6(1)$]
{Rigidity theorems of Lagrangian submanifolds in the homogeneous nearly K\"ahler $\mathbb{S}^6(1)$}

\author{Zejun Hu, Jiabin Yin and Bangchao Yin}
\address{%
School of Mathematics and Statistics, Zhengzhou University,
Zhengzhou 450001, People's Republic of China}
\email{huzj@zzu.edu.cn; welcomeyjb@163.com; mathyinchao@163.com}

\thanks{This project was supported by NSF of China, Grant Number 11771404.}

\date{}


\begin{abstract}
In this paper, we study Lagrangian submanifolds of the homogeneous nearly
K\"ahler $6$-dimensional unit sphere $\mathbb{S}^6(1)$. As the main result, 
we derive a Simons' type integral inequality in terms of the second fundamental 
form for compact Lagrangian submanifolds of $\mathbb{S}^6(1)$. Moreover, we 
show that the equality sign occurs if and only if the Lagrangian submanifold 
is either the totally geodesic $\mathbb{S}^3(1)$ or the Dillen-Verstraelen-Vrancken's 
Berger sphere $S^3$ described in J Math Soc Japan, 42: 565-584, 1990.
\end{abstract}

\maketitle

\numberwithin{equation}{section}
\section{Introduction}\label{sect:1}
It is well-known that the $6$-dimensional unit sphere $\mathbb{S}^6(1)$ with the
standard metric $g$ of constant sectional curvature $1$ admits a canonical nearly
K\"ahler structure $J$, which can be constructed by using the Cayley number system.
A $3$-dimensional Riemannian submanifold $M^3$ of $\mathbb{S}^6(1)$ is called
Lagrangian if $J(TM^3)=T^\perp M^3$, where $TM^3$ and $T^\perp M^3$ denote,
respectively, the tangent and normal bundle of $M^3$ in $\mathbb{S}^6(1)$.
Butruille \cite{B} proved that the only Riemannian homogeneous $6$-dimensional nearly K\"ahler
manifolds are $\mathbb{S}^6$, $\mathbb S^3 \times \mathbb S^3$, $\mathbb CP^3$ and
$SU(3)/U(1)\times U(1)$. However, Foscolo and Haskins \cite{F-H}
have proved the existence of at least one exotic (cohomogeneity one) nearly K\"ahler
structure on both $\mathbb{S}^6$ and $\mathbb S^3 \times \mathbb S^3$. In this paper,
we consider $\mathbb{S}^6(1)$ restricted to its canonical homogeneous nearly K\"ahler
structure.

For compact Lagrangian submanifolds of the nearly K\"ahler $\mathbb{S}^6(1)$,
the rigidity phenomena with respect to the sectional curvature $K$, the Ricci
curvature $Ric$ and the scalar curvature $\tau$ have been previously studied
in \cite{A-D-V,D-O-V-V,D-V-V1,D-V-V2,Hou,L}.

Regarding the pinching theorems for the sectional curvature, we have

\begin{theorem}[\cite{D-O-V-V,D-V-V1}]\label{thm:1.1}
Let $M^3$ be a compact Lagrangian submanifold of the nearly K\"ahler $\mathbb{S}^6(1)$
whose sectional curvatures $K$ satisfy $K>\tfrac{1}{16}$. Then $M^3$ is totally geodesic,
and thus $K=1$.
\end{theorem}

\begin{theorem}[\cite{D-V-V1,D-V-V2}]\label{thm:1.2}
Let $M^3$ be a compact Lagrangian submanifold of the nearly K\"ahler $\mathbb{S}^6(1)$
whose sectional curvatures $K$ satisfy $\tfrac1{16}\le K<\tfrac{21}{16}$, then $M^3$ is
totally geodesic with $K=1$, or $M^3$ has constant sectional curvature $\tfrac{1}{16}$.
\end{theorem}

Notice that Lagrangian submanifolds of $\mathbb{S}^6(1)$ with constant sectional
curvature were classified by Ejiri \cite{E}. Each such submanifold is either
totally geodesic or congruent to an equivariant immersion of $\mathbb{S}^3(1/16)$
in $\mathbb{S}^6(1)$ (the immersion can be realized by using harmonic polynomials
of degree $6$ and an explicitly expression is given in \cite{D-V-V1,D-V-V2}).
Also there exists another equivariant immersion of ${\rm SU}(2)$,
equipped with a suitable left invariant metric, see \cite{D-V-V2,M}, of which at
every point all sectional curvatures satisfy $\tfrac{21}{16}>K\ge\tfrac{1}{16}$.
Therefore the above-mentioned theorems are the best possible pinching results for
the sectional curvatures.

Regarding the pinching theorems for the Ricci curvature, we have

\begin{theorem}[\cite{L}]\label{thm:1.3}
Let $M^3$ be a compact Lagrangian submanifold of the nearly K\"ahler $\mathbb{S}^6(1)$ and
assume that all Ricci curvatures $Ric$ satisfy ${\rm Ric}(v)>\tfrac{53}{64}$.
Then $M^3$ is totally geodesic, and thus ${\rm Ric}=2$ on $M^3$.
\end{theorem}

An improved version of Theorem \ref{thm:1.3} was obtained by Anti\'c-Djori\'c-Vrancken \cite{A-D-V}:

\begin{theorem}[\cite{A-D-V}]\label{thm:1.4}
Let $M^3$ be a compact Lagrangian submanifold of the nearly K\"ahler $\mathbb{S}^6(1)$ and
assume that all Ricci curvatures $Ric$ satisfy ${\rm Ric}(v)\ge\tfrac{3}{4}$.
Then $M^3$ is totally geodesic.
\end{theorem}

Since a Lagrangian submanifold of the nearly K\"ahler $\mathbb{S}^6(1)$ must be minimal (\cite{E}),
the squared length $\|h\|^2$ of the second fundamental form $h$ and the scalar curvature
$\tau$ is related by, the Gauss equation, $\tau=6-\|h\|^2$. In \cite{C-D-V-V}, the authors
classified the Lagrangian submanifolds of $\mathbb{S}^6(1)$ with constant scalar curvature
that realize the Chen's inequality. As far as the pinching theorem for the scalar curvature
are concerned, we have

\begin{theorem}[\cite{Hou}]\label{thm:1.5}
Let $M^3$ be a compact Lagrangian submanifold of the nearly K\"ahler $\mathbb{S}^6(1)$.
Assume that $\|h\|^2<\tfrac52$, then $M^3$ is totally geodesic.
\end{theorem}

\begin{remark}
Although the result of Theorem \ref{thm:1.5} is not optimal, it is still significant. In fact,
it stands for a very interesting improvement of the following results: If $M^3$ is a compact minimal
submanifold of the round sphere $\mathbb{S}^6(1)$, then A. M. Li and J. M. Li \cite{L-L}
proved the result if $\|h\|^2\le2$, while J. Simons \cite{S} and Chern-do Carmo-Kobayashi
\cite{C-D-K} earlier achieved the same result provided $\|h\|^2\le9/5$.
\end{remark}

On the other hand, we noticed that next to the totally geodesic Lagrangian immersion
$\mathbb{S}^3(1)\hookrightarrow\mathbb{S}^6(1)$ for which we have $\|h\|^2=0$, the
isometric Lagrangian immersion $\mathbb{S}^3(1/16)\hookrightarrow\mathbb{S}^6(1)$
has the property that $\|h\|^2=45/8$. Thus, Theorem \ref{thm:1.5} and
Theorems \ref{thm:1.1} and \ref{thm:1.2} motivate us to consider the following problem:

\vskip 2mm\noindent
{\bf Problem}. {\it Try to characterize the compact Lagrangian submanifold of the
nearly K\"ahler $\mathbb{S}^6(1)$ whose second fundamental form $h$ has an optimal
value of length next to that of the totally geodesic one.}

\vskip 2mm

In this paper, we have solved the above problem. More specifically, for compact Lagrangian
submanifolds of $\mathbb{S}^6(1)$, we will derive an optimal Simons' type integral inequality
in terms of the second fundamental form. Our main result is the following

\vskip 2mm

\noindent{\bf Main Theorem}. {\it
Let $M^3$ be a compact Lagrangian submanifold of the nearly K\"ahler $\mathbb{S}^6(1)$.
Then it holds the Simons' type integral inequality
\begin{equation}\label{eqn:1.1}
\int_{M^3}\|h\|^2\big(\|h\|^2-\tfrac54-\tfrac32\Theta^2\big)dM\ge0,
\end{equation}
where $\Theta(p)=\max_{u\in U_pM^3}g(h(u,u),Ju)$ for $p\in M^3$.

Moreover, the equality sign in \eqref{eqn:1.1} holds if and only if
$M^3$ is either the totally geodesic $\mathbb{S}^3(1)$ with $\|h\|^2\equiv0$, or the
Dillen-Verstraelen-Vrancken's Berger sphere $\Psi(S^3)$ defined by \eqref{eqn:3.1}
which satisfies $\|h\|^2=\tfrac54+\tfrac32\Theta^2$ with $\|h\|^2\equiv\tfrac{25}8$ and
$\Theta\equiv\tfrac{\sqrt{5}}2$.}

\vskip 2mm

As direct consequence of the Main Theorem, we have
\begin{corollary}\label{cor:1.1}
Let $M^3$ be a compact Lagrangian submanifold of the nearly K\"ahler $\mathbb{S}^6(1)$.
If $\|h\|^2\le \tfrac54+\tfrac32\Theta^2$, then either $\|h\|^2\equiv0$ and $M^3$
is totally geodesic, or $\|h\|^2=\tfrac54+\tfrac32\Theta^2$ with $\|h\|^2\equiv{25}/8$ and
$\Theta\equiv\sqrt{5}/2$ and $M^3$ is the Dillen-Verstraelen-Vrancken's Berger sphere
$\Psi(S^3)$ that is defined by \eqref{eqn:3.1}.
\end{corollary}

\begin{remark}
Generalizing the observation that a parallel Lagrangian submanifold
of the nearly K\"ahler $\mathbb{S}^6(1)$ is totally geodesic in \cite{D-V}, it was shown in
\cite{ZDHVW} that, in any $6$-dimensional strict nearly K\"ahler manifold, Lagrangian
submanifolds with parallel second fundamental form are always totally geodesic.
On the other hand, M. Djori\'c and L. Vrancken \cite{D-V} considered Lagrangian submanifolds
of the nearly K\"ahler $\mathbb{S}^6(1)$ which satisfy the following condition,
namely for any tangent vector $v$ it holds
\begin{equation}\label{eqn:1.2}
g((\nabla h)(v,v,v),Jv)=0.
\end{equation}
Lagrangian submanifolds satisfying the above condition were called $J$-parallel. It is worth
pointing out that if the equality sign of \eqref{eqn:1.1} holds then $M^3$ is $J$-parallel,
and that the $J$-parallel Lagrangian submanifolds of $\mathbb{S}^6(1)$ have been classified
in \cite{D-V}. In this respect, see also \cite{H-Z} for a complete classification of the
$J$-parallel Lagrangian submanifolds of the homogeneous nearly K\"ahler manifold
$\mathbb{S}^3\times\mathbb{S}^3$.
\end{remark}

\section{The nearly K\"ahler $\mathbb{S}^6(1)$ and its Lagrangian submanifolds}\label{sect:2}

In this section, we review some aspects of the nearly K\"ahler manifold $\mathbb{S}^6(1)$ and
its Lagrangian submanifolds. More details can be found in \cite{Se} and \cite{D-V-V2,D-V}.

By considering $\mathbb R^7$ as the imaginary Cayley numbers, the Cayley multiplication
induces a vector product on $\mathbb R^7$. On $S^6:=\mathbb{S}^6(1)$ with the standard
metric $g$ we now define a $(1,1)$-tensor field $J$ by
$$
J_xU=x\times U,
$$
for $ x\in S^6$ and $U\in T_xS^6$. It is well defined (i.e., $J_xU\in T_xS^6$) and determines
an almost complex structure on $\mathbb{S}^6(1)$.
Furthermore, let $G$ be the $(2,1)$-tensor field on $S^6$ defined by
\begin{equation}\label{eqn:2.1}
G(X,Y)=(\bar\nabla_XJ)Y,
\end{equation}
where $\bar\nabla$ is the Levi-Civita connection on $\mathbb{S}^6(1)$. Then we have (cf. \cite{D-V-V2,E}):
\begin{equation}\label{eqn:2.2}
G(X,Y)+G(Y,X)=0,
\end{equation}
\begin{equation}\label{eqn:2.3}
G(X,JY)+JG(X,Y)=0,
\end{equation}
\begin{equation}\label{eqn:2.4}
g(G(X,Y),Z)+g(G(X,Z),Y)=0,
\end{equation}
\begin{equation}\label{eqn:2.5}
(\bar\nabla_XG)(Y,Z)=g(Y,JZ)X+g(X,Z)JY-g(X,Y)JZ,
\end{equation}
\begin{equation}\label{eqn:2.6}
\begin{aligned}
g(G(X,Y),G(Z,W))=&g(X,Z)g(Y,W)-g(X,W)g(Z,Y)\\
&+g(JX,Z)g(Y,JW)-g(JX,W)g(Y,JZ),
\end{aligned}
\end{equation}
where $X,Y,Z,W$ are vector fields on $S^6$. Here, \eqref{eqn:2.2} and \eqref{eqn:2.6}
imply that $(S^6,g,J)$ is a strict nearly K\"ahler manifold.

Let $x:M^3\rightarrow\mathbb{S}^6(1)$ be a Lagrangian isometric immersion. We denote
the Levi-Civita connection of $M^3$ by $\nabla$ and the normal connection in the
normal bundle $T^\perp M^3$ (defined by the orthogonal projection of $\bar\nabla$
on $T^\perp M^3$) by $\nabla^\perp$. The shape operator $A_\xi$ in the direction of a
normal vector field $\xi$ on $M^3$ and $T^\perp M^3$-valued second fundamental form $h$
are defined by the following Gauss-Weingarten formulas
\begin{equation}\label{eqn:2.7}
\bar\nabla_XY=\nabla_XY+h(X,Y),\ \
\bar\nabla_X\xi=-A_{\xi}X+\nabla^{\perp}_X{\xi},
\end{equation}
where $X,Y$ are tangent vector fields of $M^3$, and $h$ is related to $A_\xi$ by
\begin{equation}\label{eqn:2.8}
g(h(X,Y),\xi)=g(A_{\xi}X,Y).
\end{equation}

From \eqref{eqn:2.1} and \eqref{eqn:2.7} we compute that
\begin{equation}\label{eqn:2.9}
\nabla^{\perp}_X{JY}=G(X,Y)+J\nabla_XY,\ \ A_{JX}Y=-Jh(X,Y).
\end{equation}

After having the results for the nearly K\"ahler $\mathbb{S}^6(1)$, the following two lemmas have
been proved for all $6$-dimensional strict nearly K\"ahler manifold.
\begin{lemma}[\cite{E,S-S}]\label{lem:2.1}
Let $M^3$ be a Lagrangian submanifold of a $6$-dimensional strict nearly K\"ahler manifold. Then
\begin{enumerate}
\item[(1)] $M^3$ is orientable and minimal,
\item[(2)] $M^3$ has volume form $\omega(X,Y,Z)=g(G(X,Y),JZ)$,
\item[(3)] If $X,Y$ are tangent vector fields of $M^3$, then $G(X,Y)$ is a normal vector field.
\end{enumerate}
\end{lemma}

\begin{lemma}[\cite{ZDHVW}]\label{lem:2.2}
Let $M^3$ be a Lagrangian submanifold of a $6$-dimensional strict nearly K\"ahler
manifold. Then we have
$$
g((\nabla h)(W,X,Z),JY)-g((\nabla h)(W,X,Y),JZ)=g(h(W,X),G(Y,Z)),
$$
for any tangent vector fields $X, Y, Z,W$ on $M^3$.
\end{lemma}

Let $x:M^3\to\mathbb{S}^6(1)\hookrightarrow\mathbb{R}^7$ be a Lagrangian submanifold
of $\mathbb{S}^6(1)$. From now on, we agree on the following index ranges:
$$
1\leq i,j,k,l,\cdots\leq 3\ \ {\rm and}\ \ i^{*}=3+i\ \ {\rm for}\ i=1,2,3.
$$

We choose $\{e_1,e_2,e_3,e_{1^*},e_{2^*},e_{3^*}\}$ to be a local
orthonormal frame field of the tangent bundle $TS^6$ such that $e_i$
lies in $TM^3$ and $e_{i^*}=Je_i$ lies in $T^{\perp}M^3$. Let
$\{\omega_1,\omega_2,\omega_3,\omega_{1^*},\omega_{2^*},\omega_{3^*}\}$
be the associated dual frame field so that restricted to $M^3$ it holds that
$\omega_{1^*}=\omega_{2^*}=\omega_{3^*}=0$. With respect to $\{e_1,e_2,e_3,e_{1^*},e_{2^*},e_{3^*}\}$,
let $\omega_{ij}$ and $\omega_{i^*j^*}$ denote the
connection $1$-forms of $TM^3$ and $T^\perp M^3$, respectively.
Then the structure equations of $x:M^3\to\mathbb{S}^6(1)$ are:
\begin{equation}\label{eqn:2.10}
\left\{
\begin{aligned}
&dx=\sum_i\omega_ie_i,\\
&de_i=\sum_j\omega_{ij}e_j+\sum_{j,k} h^{k^*}_{ij}\omega_je_{k^*}-\omega_ix,\ \ \omega_{ij}+\omega_{ji}=0,\\
&de_{i^*}=-\sum_{j,k} h^{i^*}_{jk}\omega_je_k+\sum_j\omega_{i^*j^*}e_{j^*},\ \ \omega_{i^*j^*}+\omega_{i^*j^*}=0,
\end{aligned}
\right.
\end{equation}
where $h_{ij}^{k^*}=h_{ji}^{k^*}=h_{ik}^{j^*}$ for any $i,j,k$, and
$h=\sum_{i,j,k}h_{ij}^{k^*}\omega_i\omega_je_{k^*}$. Taking
exterior differentiation of \eqref{eqn:2.10} we get
\begin{equation}\label{eqn:2.11}
\left\{
\begin{aligned}
&d\omega_i=\sum_j\omega_{ij}\wedge\omega_j,\\
&d\omega_{ij}-\sum_k\omega_{ik}\wedge\omega_{kj}:=-\tfrac12\sum_{k,l}R_{ijkl}\omega_k\wedge\omega_l,\\
&d\omega_{i^*j^*}-\sum_k\omega_{i^*k^*}\wedge\omega_{k^*j^*}:=-\tfrac12\sum_{k,l}R_{i^*j^*kl}\omega_k\wedge\omega_l,\\
&\sum_lh_{ij,l}^{k^*}\omega_l:=dh_{ij}^{k^*}+\sum_lh_{il}^{k^*}\omega_{lj}+\sum_lh_{lj}^{k^*}\omega_{li}+\sum_lh_{ij}^{l^*}\omega_{l^*k^*},
\end{aligned}
\right.
\end{equation}
where $R_{ijkl},\ R_{i^*j^*kl}$ and $h^{k^*}_{ij,l}$ are components of the curvature
tensor of the tangent bundle, the normal bundle and the first covariant derivative of the second
fundamental form of $M^3$, and they satisfy the Gauss-Codazzi-Ricci equations:
\begin{equation}\label{eqn:2.12}
R_{ijkl}=\delta_{ik}\delta_{jl}-\delta_{il}\delta_{jk}+\sum_p(h^{p^*}_{ik}h^{p^*}_{jl}-
h^{p^*}_{il}h^{p^*}_{jk}),
\end{equation}
\begin{equation}\label{eqn:2.13}
h^{k^*}_{ij,l}=h^{k^*}_{il,j},
\end{equation}
\begin{equation}\label{eqn:2.14}
R_{i^*j^*kl}=\sum_p(h^{p^*}_{ik}h^{p^*}_{jl}-
h^{p^*}_{il}h^{p^*}_{jk}).
\end{equation}

From \eqref{eqn:2.12}, the Ricci curvature $R_{ij}$ and the scalar curvature $\tau$ of $M^3$ satisfy
\begin{equation}\label{eqn:2.15}
R_{ij}=3\delta_{ij}-\sum_{k,p}h^{p^*}_{ik}h^{p^*}_{kj},\ \ \tau=6-\|h\|^2,
\end{equation}
where $\|h\|^2=\sum_{i,j,k}(h^{k^*}_{ij})^2$ is the squared length of the second fundamental form.

Exterior differentiation of the last equation of \eqref{eqn:2.11} we get the Ricci identity
\begin{equation}\label{eqn:2.16}
h^{p^*}_{ij,kl}-h^{p^*}_{ij,lk}=\sum_mh^{p^*}_{mi}R_{mjkl}+\sum_mh^{p^*}_{mj}R_{mikl}+\sum_mh^{m^*}_{ij}R_{m^*p^*kl},
\end{equation}
where, $h^{p^*}_{ij,kl}$ is the components of the second covariant derivative of $h$:
$$
\sum_lh_{ij,kl}^{p^*}\omega_l:=dh_{ij,k}^{p^*}+\sum_lh_{lj,k}^{p^*}\omega_{li}+\sum_lh_{il,k}^{p^*}\omega_{lj}
+\sum_lh_{ij,l}^{p^*}\omega_{lk}+\sum_lh_{ij,k}^{l^*}\omega_{l^*p^*}.
$$
\numberwithin{equation}{section}
\section{Dillen-Verstraelen-Vrancken's Berger sphere in $\mathbb{S}^6(1)$}\label{sect:3}

Consider the unit sphere $S^3:=\{(y_1,y_2,y_3,y_4)\in\mathbb{R}^4\,|\,y_1^2+y_2^2+y_3^2+y_4^2=1\}$
in $\mathbb{R}^4$.
There are many Lagrangian immersions from the topological three-sphere $S^3$
into the nearly K\"ahler unit $6$-sphere that have nice properties. Indeed, besides that of
constant sectional curvature appeared in Theorem \ref{thm:1.2}, immersions of Berger $3$-spheres
are also introduced and geometrically characterized in \cite{D-V-V2} and \cite{C-D-V-V} (see also \cite{L-W}).
For our purpose, we particularly mention that, in \cite{D-V-V2} (cf. also \cite{D-V} and \cite{L-W}),
Dillen, Verstraelen and Vrancken constructed an embedding from the topological three-sphere
into the nearly K\"ahler unit $6$-sphere, defined by
\begin{equation}\label{eqn:3.1}
\Psi: S^3\to\mathbb{S}^6(1):\ (y_1,y_2,y_3,y_4)\mapsto(x_1,x_2,x_3,x_4,x_5,x_6,x_7),
\end{equation}
where
$$
\left\{
\begin{aligned}
&x_1=\tfrac19(5y_1^2+5y_2^2-5y_3^2-5y_4^2+4y_1),\ \ x_2=-\tfrac23y_2,\\
&x_3=\tfrac{2\sqrt{5}}9(y_1^2+y_2^2-y_3^2-y_4^2-y_1),\ \ \ \ \ \ \,
       x_4=\tfrac{\sqrt{3}}{9\sqrt{2}}(-10y_1y_3-2y_3-10y_2y_4),\\
&x_5=\tfrac{\sqrt{3}\sqrt{5}}{9\sqrt{2}}(2y_1y_4-2y_4-2y_2y_3),\ \ \ \ \ \ \ \,
       x_6=\tfrac{\sqrt{3}\sqrt{5}}{9\sqrt{2}}(2y_1y_3-2y_3+2y_2y_4),\\
&x_7=\tfrac{\sqrt{3}}{9\sqrt{2}}(10y_1y_4+2y_4-10y_2y_3).
\end{aligned}
\right.
$$

To make calculation of the mapping $\Psi: S^3\to\mathbb{S}^6(1)$, let
$X_1,\ X_2,\ X_3$ be the vector fields on $S^3$, defined by
$$
\left\{
\begin{aligned}
&X_1(y_1,y_2,y_3,y_4)=(y_2,-y_1,y_4,-y_3),\\
&X_2(y_1,y_2,y_3,y_4)=(y_3,-y_4,-y_1,y_2),\\
&X_3(y_1,y_2,y_3,y_4)=(y_4,y_3,-y_2,-y_1).
\end{aligned}
\right.
$$
Then $X_1, X_2$ and $X_3$ form a basis of tangent vector fields
to $S^3$, and it holds that $[X_1, X_2]=2X_3,\ [X_2, X_3]=2X_1$ and $[X_3, X_1]=2X_2$.

We define a Berger metric $\langle\cdot, \cdot\rangle$ on $S^3$ such that $X_1, X_2$ and $X_3$ are orthogonal and
such that $\langle X_1, X_1\rangle=4/9$ and $\langle X_2, X_2\rangle=\langle X_3, X_3\rangle=8/3$. Then
$$
E_1=\tfrac32X_1,\ \ E_2=\tfrac{\sqrt{3}}{2\sqrt{2}}X_2,\ \ E_3=-\tfrac{\sqrt{3}}{2\sqrt{2}}X_3
$$
form an orthonormal frame field on $(S^3,\langle\cdot, \cdot\rangle)$. Moreover, direct calculations give the following results.
\begin{lemma}[\cite{D-V-V2}]\label{lem:3.1}
The curvature tensor of the Berger sphere $(S^3,\langle\cdot, \cdot\rangle)$ has the following expression
$$
\begin{aligned}
\langle R(X, Y)W, Z\rangle=&\tfrac1{16}(\langle X, Z\rangle\langle Y, W\rangle-\langle X, W\rangle\langle Y, Z\rangle)\\
&+\tfrac{20}{16}(\langle X^{\perp}, Z^{\perp}\rangle\langle Y^{\perp}, W^{\perp}\rangle-\langle X^{\perp},
   V^{\perp}\rangle\langle Y^{\perp}, Z^{\perp}\rangle),
\end{aligned}
$$
where $V^\perp$ denotes the orthogonal complement of a vector $V$ with respect to $E_1$. Moreover, $(S^3,\langle\cdot, \cdot\rangle)$
has constant scalar curvature $23/16$.
\end{lemma}
\begin{lemma}[\cite{D-V-V2,D-V}]\label{lem:3.2}
The above mapping $\Psi: S^3\to\mathbb{S}^6(1)$ is an isometric Lagrangian embedding from $(S^3,\langle\cdot, \cdot\rangle)$
into $\mathbb{S}^6(1)$. Moreover, with respect to the globally defined
orthonormal tangent vector fields $\{E_1, E_2, E_3\}$, it holds that $G(E_2, E_3)=JE_1$, and the second fundamental form
$h$ of $\Psi: S^3\to\mathbb{S}^6(1)$ takes the following form
$$
\left\{
\begin{aligned}
&h(E_1,E_1)=\tfrac{\sqrt{5}}2JE_1,\ \ \ \  h(E_1,E_2)=-\tfrac{\sqrt{5}}4JE_2,\ \ h(E_1,E_3)=-\tfrac{\sqrt{5}}4JE_3,\\
&h(E_2,E_2)=-\tfrac{\sqrt{5}}4JE_1,\ \ h(E_3,E_3)=-\tfrac{\sqrt{5}}4JE_1,\ \ h(E_2,E_3)=0.
\end{aligned}
\right.
$$
\end{lemma}

\begin{remark}\label{rem:3.1}~
\begin{enumerate}
\item[(1)] Let $\sigma$ be any plane in the tangent space of $S^3$. Then we have an
orthonormal basis $\{X, Y\}$ of $\sigma$ such that $X=\cos\theta E_2+\sin\theta E_3$ and $Y=\sin\varphi E_1
-\cos\varphi\sin\theta E_2+\cos\varphi\cos\theta E_3$, where $\theta, \varphi\in\mathbb{R}$.
Thus the sectional curvature of the plane $\sigma$ is given by
$R(X, Y, Y, X)=K(\sigma)=1/16+20/16\cos^2\varphi$. It follows that $1/16\leqq K(\sigma)\leqq 21/16$,
where $1/16$ is attained for every plane
which contains $E_1$, and where $21/16$ is attained only for the plane spanned by $E_2$ and $E_3$.

\item[(2)] Lemma \ref{lem:3.2} implies that the second fundamental form of the Lagrangian embedding
$\Psi: S^3\to\mathbb{S}^6(1)$ has constant squared norm. Indeed, it holds that
$\|h\|^2(p)=\tfrac{25}8$, $\Theta(p)=\max_{u\in U_pS^3}\langle h(u,u),Ju\rangle=\tfrac{\sqrt{5}}2$
for any $p\in S^3$.

\item[(3)] Due to Lemmas \ref{lem:3.1} and \ref{lem:3.2}, we will call the embedding $\Psi: S^3\to\mathbb{S}^6(1)$
defined by \eqref{eqn:3.1} as the {\it Dillen-Verstraelen-Vrancken's Berger sphere}.

\end{enumerate}
\end{remark}

\numberwithin{equation}{section}
\section{Lemmas and Proof of the Main Theorem}\label{sect:4}
First, thanks to that Lagrangian submanifolds of the nearly K\"ahler $\mathbb{S}^6(1)$
are minimal, and applying for the Gauss-Codazzi-Ricci equations \eqref{eqn:2.12}--\eqref{eqn:2.14}
and the Ricci identity \eqref{eqn:2.16}, we have the following well known result.
\begin{lemma}[\cite{C-D-K,L-L}]\label{lem:4.1}
Let $M^3$ be a Lagrangian submanifold of the nearly K\"ahler $\mathbb{S}^6(1)$.
Then, in terms the notations in section \ref{sect:2} and put $H_i=(h_{jk}^{i^*})$, we have the following
formula for the Laplacian of $\|h\|^2$:
\begin{equation}\label{eqn:4.1}
\frac12\Delta\|h\|^2=\sum_{i,j,k}(h_{ij,k}^{l^*})^2+3\|h\|^2-\sum_{i,j}N(H_iH_j-H_jH_i)-\sum_{i,j}S_{ij}^2.
\end{equation}
Here, $S_{ij}={\rm trace}(H_iH_j)$ and $N(A)={\rm trace}(AA^t)=\sum_{i,j}(a_{ij})^2$ for $A=(a_{ij})$.
\end{lemma}

Next, to calculate the invariant $\sum_{i,j}N(H_iH_j-H_jH_i)+\sum_{i,j}S_{ij}^2$, we will choose
a canonical orthonormal bases following the standard way of N. Ejiri \cite{E}.

Let $M^3$ be a Lagrangian submanifold of the nearly K\"ahler $\mathbb{S}^6(1)$.
Let $UM^3$ be the unit tangent bundle over $M^3$ such that
$U_qM^3=\{u\in T_qM^3\mid g(u,u)=1\}$ for any $q\in M^3$.
We define a function $f_q$ on $U_qM^3$ by $f_q(u)=g(h(u,u),Ju)$.
Since $U_qM^3$ is compact, there is an element $e_1\in U_qM^3$ such that
$f_q(e_1)=\max_{u\in U_qM^3}f_q(u)$. Actually, we have the following lemma.

\begin{lemma}[\cite{A-D-V,D-V}]\label{lem:4.2}
Let $M^3$ be a Lagrangian submanifold of the nearly K\"ahler $\mathbb{S}^6(1)$.
Then, for all $q\in M^3$, there exists an orthonormal basis $\{e_1,e_2,e_3\}$ of $T_qM^3$ such that
\begin{equation}\label{eqn:4.2}
\left\{
\begin{aligned}
&h(e_1,e_1)=(\lambda_1+\lambda_2)Je_1,\ \ h(e_1,e_2)=-\lambda_1Je_2,\ \ h(e_1,e_3)=-\lambda_2Je_3,\\
&h(e_2,e_2)=-\lambda_1Je_1+\mu_1Je_2+\mu_2Je_3,\ \ h(e_2,e_3)=\mu_2Je_2-\mu_1Je_3,\\
&h(e_3,e_3)=-\lambda_2Je_1-\mu_1Je_2-\mu_2Je_3,
\end{aligned}
\right.
\end{equation}
where
\begin{equation}\label{eqn:4.3}
\left\{
\begin{aligned}
&\lambda_1+\lambda_2=\max_{u\in U_qM^3}f_q(u)\ge0,\\
&3\lambda_1+\lambda_2\ge 0,\ \ 3\lambda_2+\lambda_1\ge 0,\\
&-(\lambda_1+\lambda_2)\le\mu_1,\ \mu_2\le\lambda_1+\lambda_2.
\end{aligned}
\right.
\end{equation}
\end{lemma}

\begin{lemma}\label{lem:4.3}
If \eqref{eqn:4.2} holds, then by notations of Lemma \ref{lem:4.1} we have
\begin{equation}\label{eqn:4.4}
\|h\|^2=\sum_{i,j,k}(h^{k^*}_{ij})^2=4\lambda_1^2+4\lambda_2^2+2\lambda_1\lambda_2+4\mu_1^2+4\mu_2^2,
\end{equation}
\begin{equation}\label{eqn:4.5}
\begin{aligned}
\sum_{i,j}&N(H_iH_j-H_jH_i)+\sum_{i,j}(S_{ij})^2\\
&=24\lambda_1^4+24\lambda_1^3\lambda_2+24\lambda_1^2\lambda_2^2+24\lambda_1\lambda_2^3+24\lambda_2^4\\
&\ \ \ +18(\lambda_1^2+\lambda_2)(\mu_1^2+\mu_2^2)-36\lambda_1\lambda_2(\mu_1^2+\mu_2^2)+24(\mu_1^2+\mu_2^2)^2.
\end{aligned}
\end{equation}
\end{lemma}
\begin{proof}
If \eqref{eqn:4.2} holds, then we can write $H_k=(h^{k^*}_{ij})$ in more explicit form:
\begin{equation}\label{eqn:4.6}
H_1=
\begin{pmatrix}
\lambda_1+\lambda_2&0&0\\
0&-\lambda_1&0\\
0&0&-\lambda_2
\end{pmatrix},
\end{equation}
\begin{equation}\label{eqn:4.7}
H_2=
\begin{pmatrix}
0&-\lambda_1&0\\
-\lambda_1&\mu_1&\mu_2\\
0&\mu_2&-\mu_1
\end{pmatrix},
\end{equation}
\begin{equation}\label{eqn:4.8}
H_3=
\begin{pmatrix}
0&0&-\lambda_2\\
0&\mu_2&-\mu_1\\
-\lambda_2&-\mu_1&-\mu_2
\end{pmatrix}.
\end{equation}

From \eqref{eqn:4.6}--\eqref{eqn:4.8}, we have the following computations
\begin{equation}\label{eqn:4.9}
H_1H_2-H_2H_1=
\begin{pmatrix}
0&-\lambda_1(2\lambda_1+\lambda_2)&0\\
\lambda_1(2\lambda_1+\lambda_2)&0&(\lambda_2-\lambda_1)\mu_2\\
0&(\lambda_1-\lambda_2)\mu_2&0
\end{pmatrix},
\end{equation}
\begin{equation}\label{eqn:4.10}
H_1H_3-H_3H_1=
\begin{pmatrix}
0&0&-\lambda_2(\lambda_1+2\lambda_2)\\
0&0&(\lambda_1-\lambda_2)\mu_1\\
\lambda_2(\lambda_1+2\lambda_2)&(\lambda_2-\lambda_1)\mu_1&0
\end{pmatrix},
\end{equation}
\begin{equation}\label{eqn:4.11}
H_2H_3-H_3H_2=
\begin{pmatrix}
0&(\lambda_2-\lambda_1)\mu_2&(\lambda_1-\lambda_2)\mu_1\\
(\lambda_1-\lambda_2)\mu_2&0&\lambda_1\lambda_2-2\mu_1^2-2\mu_2^2\\
(\lambda_2-\lambda_1)\mu_1&2\mu^2_1+2\mu^2_2-\lambda_1\lambda_2&0
\end{pmatrix}.
\end{equation}
It follows that
\begin{equation}\label{eqn:4.12}
\begin{aligned}
2N(H_1H_2-H_2H_1)=16\lambda_1^4
+16\lambda_1^3\lambda_2&+4\lambda_1^2\lambda_2^2+4\lambda_1^2\mu_2^2\\
&-8\lambda_1\lambda_2\mu_2^2+4\lambda_2^2\mu_2^2,
\end{aligned}
\end{equation}
\begin{equation}\label{eqn:4.13}
\begin{aligned}
2N(H_1H_3-H_3H_1)=4\lambda_1^2 \lambda_2^2 + 16\lambda_1
\lambda_2^3&+16 \lambda_2^4+4 \lambda_1^2 \mu_1^2\\
&-8\lambda_1\lambda_2\mu_1^2 + 4 \lambda_2^2\mu_1^2,
\end{aligned}
\end{equation}
\begin{equation}\label{eqn:4.14}
\begin{aligned}
2N(H_2H_3-H_3H_2)=4(\lambda_1-\lambda_2)^2(\mu_1^2+\mu_2^2)+4(\lambda_1\lambda_2-2\mu_1^2-2\mu_2^2)^2.
\end{aligned}
\end{equation}

Next, by direct calculation of $S_{ij}=\sum_{k,l}h^{i^*}_{kl}h^{j^*}_{kl}$, we get
$$
\begin{aligned}
\sum_{i,j}(S_{ij})^2=\,4(\lambda_1^2&+\lambda_2^2+\lambda_1\lambda_2)^2+4(\lambda_1^2+\mu_1^2+\mu_2^2)^2\\[-3mm]
             &+4(\lambda_2^2+\mu_1^2+\mu_2^2)^2+2(\lambda_1-\lambda_2)^2(\mu_1^2+\mu_2^2).
\end{aligned}
$$

From the above computations we immediately verify \eqref{eqn:4.4} and \eqref{eqn:4.5}.
\end{proof}

Next, for a Lagrangian submanifold $M^3$ of the nearly K\"ahler $\mathbb{S}^6(1)$, we introduce a
$T^{\perp}M^3$-valued tensor $\mathbb{T}:TM^3\times TM^3\times TM^3\to T^{\perp}M^3$ by
\begin{equation}\label{eqn:4.15}
\mathbb{T}(X,Y,Z)=(\nabla h)(X,Y,Z)-F(X,Y,Z)
\end{equation}
where $F(X,Y,Z)=\tfrac14\big[G(X,A_{JZ}Y)+G(Y,A_{JX}Z)+G(Z,A_{JY}X)\big]$.

The tensor $\mathbb{T}$ has important properties that we state as the following lemmas.

\begin{lemma}\label{lem:4.4}
Let $M^3$ be a Lagrangian submanifold of the nearly K\"ahler $\mathbb{S}^6(1)$. Then we have
\begin{equation}\label{eqn:4.16}
\sum_{i,j,k,l}(h_{ij,k}^{l^*})^2=\parallel\nabla h\parallel^2=\,\parallel\mathbb{T}\parallel^2+\tfrac34\|h\|^2.
\end{equation}
\end{lemma}
\begin{proof}
Let $\{e_1,e_2,e_3\}$ be a local orthonormal basis of the tangent bundle of $M^3$ as assumed
in section \ref{sect:2}. From \eqref{eqn:2.9}, we have $A_{Je_i}e_j=-Jh(e_i,e_j)=\sum_kh_{ij}^{k^*}e_k$.
It follows that
$$
F(e_i,e_j,e_k)=\tfrac14\sum_l\big[h_{jk}^{l^*}G(e_i,e_l)+h_{ik}^{l^*}G(e_j,e_l)+h_{ij}^{l^*}G(e_k,e_l)\big].
$$

Then, by using the minimality of $M^3$ and \eqref{eqn:2.6}, which gives that
\begin{equation}\label{eqn:4.17}
g(G(e_i,e_j),G(e_k,e_l))=\delta_{ik}\delta_{jl}-\delta_{il}\delta_{jk},
\end{equation}
we can easily verify by direct calculations that
\begin{equation}\label{eqn:4.18}
\|F\|^2=\sum_{i,j,k}g(F(e_i,e_j,e_k),F(e_i,e_j,e_k))=\tfrac34\|h\|^2.
\end{equation}

Next, by definition $(\nabla h)(e_k,e_i,e_j)=\sum_lh_{ij,k}^{l^*}Je_l$, applying Lemma \ref{lem:2.2}
and \eqref{eqn:2.13} we get
\begin{equation}\label{eqn:4.19}
h^{l^*}_{ik,j}-h^{j^*}_{ik,l}=\sum_ph^{p^*}_{ik}g(Je_p,G(e_l,e_j)).
\end{equation}

Using \eqref{eqn:2.4}, \eqref{eqn:2.13} and \eqref{eqn:4.19}, we have the following calculation:
\begin{equation}\label{eqn:4.20}
\begin{aligned}
&\sum_{i,j,k}g((\nabla h)(e_i,e_j,e_k),F(e_i,e_j,e_k))\\
&=\frac14\sum_{i,j,k,l,p}h_{jk,i}^{l^*}g(Je_l,h_{jk}^{p^*}G(e_i,e_p)+h_{ik}^{p^*}G(e_j,e_p)+h_{ij}^{p^*}G(e_k,e_p))\\
&=\frac18\sum_{i,j,k,l,p}\Big[h_{jk}^{p^*}(h_{jk,i}^{l^*}-h_{jk,l}^{i^*})g(Je_l,G(e_i,e_p))\\
&\hspace{2cm}+h_{ik}^{p^*}(h_{ik,j}^{l^*}-h_{ik,l}^{j^*})g(Je_l,G(e_j,e_p))\\
&\hspace{2cm}+h_{ij}^{p^*}(h_{ij,k}^{l^*}-h_{ij,l}^{k^*})g(Je_l,G(e_k,e_p))\Big]\\
&=\frac18\sum_{i,j,k,l,p,m}\Big[h_{jk}^{p^*}h_{jk}^{m^*}g(Je_m,G(e_l,e_i))g(Je_l,G(e_i,e_p))\\
&\hspace{2cm}+h_{ik}^{p^*}h_{ik}^{m^*}g(Je_m,G(e_l,e_j))g(Je_l,G(e_j,e_p))\\
&\hspace{2cm}+h_{ij}^{p^*}h_{ij}^{m^*}g(Je_m,G(e_l,e_k))g(Je_l,G(e_k,e_p))\Big]
\end{aligned}
\end{equation}

Now, by using \eqref{eqn:2.3} and \eqref{eqn:2.4}, we have
\begin{equation}\label{eqn:4.21}
\sum_lg(Je_m,G(e_l,e_i))g(Je_l,G(e_i,e_p))=g(G(e_i,e_m),G(e_i,e_p)).
\end{equation}

Combining \eqref{eqn:4.21} and \eqref{eqn:4.17}, then inserting the results into \eqref{eqn:4.20}, we get
\begin{equation}\label{eqn:4.22}
\sum_{i,j,k}g((\nabla h)(e_i,e_j,e_k),F(e_i,e_j,e_k))=\tfrac34\|h\|^2.
\end{equation}

From \eqref{eqn:4.18}, \eqref{eqn:4.22} and the fact
\begin{equation}\label{eqn:4.23}
\|\mathbb{T}\|^2=\|\nabla h\|^2+\|F\|^2-2\sum_{i,j,k}g((\nabla h)(e_i,e_j,e_k),F(e_i,e_j,e_k)),
\end{equation}
we finally verify the assertion \eqref{eqn:4.16}.
\end{proof}

\begin{lemma}\label{lem:4.5}
A Lagrangian submanifold of the nearly K\"ahler $\mathbb{S}^6(1)$ satisfies $\mathbb{T}=0$
if and only if it is $J$-parallel, namely \eqref{eqn:1.2} holds.
\end{lemma}
\begin{proof}
By using \eqref{eqn:2.2}--\eqref{eqn:2.4} and \eqref{eqn:2.9}, we get the calculations:
$$
\begin{aligned}
g(JW,G(Z,A_{JY}X)&=-g(JW,G(Z,Jh(X,Y))=g(JW,JG(Z,h(X,Y))\\
&=g(W,G(Z,h(X,Y))=-g(h(X,Y),G(Z,W)).
\end{aligned}
$$
It follows that
$$
\begin{aligned}
4g(\mathbb{T}(X,Y,Z),JW)=\,&4g((\nabla h)(X,Y,Z),JW)-g(G(X,A_{JZ}Y),JW)\\
                           &-g(G(Y,A_{JX}Z),JW)-g(G(Z,A_{JY}X),JW)\\
                        =\,&4g((\nabla h)(X,Y,Z),JW)+g(h(Y,Z),G(X,W))\\
                           &+g(h(X,Z),G(Y,W))+g(h(X,Y),G(Z,W)).
\end{aligned}
$$
Therefore, $\mathbb{T}=0$ if and only if
$$
\begin{aligned}
4g((\nabla h)&(X,Y,Z),JW)+g(h(Y,Z),G(X,W))\\
&+g(h(X,Z),G(Y,W))+g(h(X,Y),G(Z,W))=0.
\end{aligned}
$$
This is equivalent to that the submanifold is $J$-parallel (cf. (24) of \cite{D-V}).
\end{proof}

Lemma \ref{lem:4.5} allows us to apply for Theorem A and Theorem 1 of \cite{D-V}
so that we can obtain the following
\begin{lemma}\label{lem:4.6}
Let $M^3$ be a Lagrangian submanifold of the nearly K\"ahler $\mathbb{S}^6(1)$.
If $M^3$ satisfies $\mathbb{T}=0$, then, for each point $q\in M^3$, there exists
an orthonormal basis $\{e_1,e_2,e_3\}$ of $T_qM^3$ such that either
\begin{enumerate}
\item[(a)] $h(e_1,e_1)=h(e_2,e_2)=h(e_3,e_3)=h(e_1,e_2)=h(e_1,e_3)=h(e_2,e_3)=0$,
i.e., $M^3$ is totally geodesic with $K=1$; or

\item[(b)] $h(e_1,e_1)=\tfrac{\sqrt{5}}2Je_1$,\ \ $h(e_1,e_2)=-\tfrac{\sqrt{5}}4Je_2$,\ \ $h(e_1,e_3)=-\tfrac{\sqrt{5}}4Je_3$,

\noindent $h(e_2,e_2)=-\tfrac{\sqrt{5}}4Je_1+\tfrac{\sqrt{10}}4Je_2$,\ \ $h(e_3,e_3)=-\tfrac{\sqrt{5}}4Je_1-\tfrac{\sqrt{10}}4Je_2$,

\noindent $h(e_2,e_3)=-\tfrac{\sqrt{10}}4Je_3$.

\noindent Moreover, $M^3$ has constant sectional curvature $\tfrac1{16}$; or

\item[(c)] $h(e_1,e_1)=\tfrac{\sqrt{5}}2Je_1$,\ \ $h(e_1,e_2)=-\tfrac{\sqrt{5}}4Je_2$, $h(e_1,e_3)=-\tfrac{\sqrt{5}}4Je_3$,

\noindent $h(e_2,e_2)=-\tfrac{\sqrt{5}}4Je_1$,\ \ $h(e_3,e_3)=-\tfrac{\sqrt{5}}4Je_1$,\ \ $h(e_2,e_3)=0$.

\noindent Moreover, $M^3$ is locally congruent to Dillen-Verstraelen-Vrancken's Berger
sphere $\Psi(S^3)$, defined by \eqref{eqn:3.1}.
\end{enumerate}
\end{lemma}

\vskip 2mm
\noindent{\bf The Completion of Main Theorem's Proof.}~

Let $M^3$ be a compact Lagrangian submanifold of the nearly K\"ahler $\mathbb{S}^6(1)$.
Now, we apply for Lemma \ref{lem:4.2} and make calculation at an arbitrary fixed point
$q\in M^3$ with the orthonormal basis $\{e_1,e_2,e_3\}$ of $T_qM^3$. Then, from
Lemmas \ref{lem:4.1}, \ref{lem:4.3} and \ref{lem:4.4}, we have
\begin{equation}\label{eqn:4.24}
\begin{aligned}
\frac12\Delta\|h\|^2=\,&\|\mathbb{T}\|^2+\tfrac{15}4\|h\|^2-\sum_{i,j}N(H_iH_j-H_jH_i)-\sum_{i,j}S_{ij}^2\\
=\,&\|\mathbb{T}\|^2+\tfrac{15}4\|h\|^2-\Big[24\lambda_1^4+24\lambda_1^3\lambda_2+24\lambda_1^2\lambda_2^2+ 24\lambda_1\lambda_2^3 + 24\lambda_2^4 \\
&\hspace{5mm}+18(\lambda_1^2+\lambda_2)(\mu_1^2+\mu_2^2)-36\lambda_1\lambda_2(\mu_1^2+\mu_2^2) +24(\mu_1^2 + \mu_2^2)^2\Big]\\
=\,&\|\mathbb{T}\|^2+\tfrac{15}4\|h\|^2-3\|h\|^4+24(\lambda_1^4+\lambda_2^4+\lambda_1\lambda_2^3+\lambda_1^3\lambda_2)+84\lambda_1^2\lambda_2^2\\
&\hspace{36mm}+78(\lambda_1^2+\lambda_2^2)(\mu_1^2+\mu_2^2)+24(\mu_1^2+\mu_2^2)^2\\
=\,&\|\mathbb{T}\|^2+\tfrac{15}4\|h\|^2-3\|h\|^4+\tfrac92(\lambda_1+\lambda_2)^2\|h\|^2+24(\mu_1^2+\mu_2^2)^2\\
&\hspace{36mm}+3(\lambda_1-\lambda_2)^2(2\lambda_1^2+2\lambda_2^2-3\lambda_1\lambda_2)\\
&\hspace{36mm}+12\big(5\lambda_1^2+5\lambda_2^2+4\lambda_1\lambda_2\big)(\mu_1^2+\mu_2^2).
\end{aligned}
\end{equation}

Noticing that $\lambda_1+\lambda_2=\max_{u\in UM^3}g(h(u,u),Ju)=\Theta$,
from Lemma \ref{lem:4.2}, \eqref{eqn:4.24} and the arbitrariness of $q\in M^3$, 
by applying for the divergence theorem, we get
\begin{equation}\label{eqn:4.25}
0=\int_{M^3}\frac12\Delta\|h\|^2dM\ge3\int_{M^3}\|h\|^2\big(\tfrac54+\tfrac32\Theta^2-\|h\|^2\big)dM.
\end{equation}
The equality sign in \eqref{eqn:4.25} holds if and only if $\mathbb{T}=0$ and that, either $M^3$ is totally
geodesic, or $\mu_1=\mu_2=0$ and $\lambda_1=\lambda_2\not=0$ on $M^3$. In the latter case, according to
Lemma \ref{lem:4.6}, $M^3$ is locally congruent to the Dillen-Verstraelen-Vrancken's Berger sphere $\Psi(S^3)$,
defined by \eqref{eqn:3.1}. It follows from Lemma \ref{lem:3.2} that $\|h\|^2\equiv{25}/8$ and
$\Theta\equiv\sqrt{5}/2$. This shows that $\|h\|^2=\tfrac54+\tfrac32\Theta^2$.\qed

\vskip 2mm

Finally, in conclusion we state the following locally rigidity theorem which is of independent meaning.
\begin{theorem}\label{thm:4.1}
Let $M^3$ be a Lagrangian submanifold of the nearly K\"ahler $\mathbb{S}^6(1)$. Then it holds that
\begin{equation}\label{eqn:4.26}
\|\nabla h\|^2\ge\tfrac34\|h\|^2.
\end{equation}

Moreover, \eqref{eqn:4.26} holds identically on $M^3$ if and only if one of the 
following three cases occurs:
\begin{enumerate}
\item[(a)] $M^3$ is totally geodesic $(K=1\ {\rm and}\ h=0)$, or
\item[(b)] $M^3$ has constant sectional curvature $\tfrac1{16}$ and $\|h\|^2=\tfrac{45}8$, or
\item[(c)] $M^3$ is locally congruent to an open part of the Dillen-Verstraelen-Vrancken's 
Berger sphere $\Psi(S^3)$ defined by \eqref{eqn:3.1} with $\|h\|^2=\tfrac{25}8$.
\end{enumerate}
\end{theorem}
\begin{proof}
This is a direct consequence of Theorem A of \cite{D-V}, Lemmas \ref{lem:4.4} and \ref{lem:4.5}.
\end{proof}

\vskip1cm

\end{document}